\documentclass[a4paper]{article}

\usepackage{
amsmath,
amsthm,
amscd,
amssymb,
}
\usepackage{comment}
\usepackage{hyperref}
\usepackage{tikz}
\usetikzlibrary{positioning,arrows,calc}
\usepackage{xspace}

\usepackage{graphicx}

\usepackage{url}

\theoremstyle{plain}
\newtheorem{lemma}{Lemma}
\newtheorem{definition}{Definition}
\newtheorem{corollary}{Corollary}

\newtheorem{theorem}{Theorem}

\newtheorem{question}{Question}

\newtheoremstyle{derp}
{3pt}
{3pt}
{}
{}
{\upshape}
{:}
{.5em}
{}
\theoremstyle{derp}

\newcommand{\Z}{\mathbb{Z}}

\newcommand{\N}{\mathbb{N}}

\newcommand{\Aut}{\mathrm{Aut}}

\title{Minimal sofic shift on a group that is not finitely-generated}

\author{
Ville Salo \\
vosalo@utu.fi
}

\begin{document}
\maketitle

\begin{abstract}
We prove that there exists a group which is not finitely generated, but admits a minimal sofic shift. This answers a question of Doucha, Melleray and Tsankov. The group is of the form $(F_4 \times F_2) \rtimes F_{\infty}$. The construction itself is based on simulation theory and properties of Thompson's~$V$.
\end{abstract}

\section{Introduction}

Let $G$ be a countable group and $A$ a finite set (called the alphabet). The \emph{full shift} is the Cantor set $A^G$ where $G$ acts continuously by $gx_h = x_{g^{-1}h}$. \emph{Subshifts} are closed $G$-invariant subsets of a full shift. \emph{Subshifts of finite type} or \emph{SFTs} are subshifts defined by removing all points whose $G$-orbit intersects a particular clopen set. \emph{Sofic shifts} are subshifts that are images of SFTs under continuous shift-commuting surjective functions; such an SFT is called the sofic shift's \emph{SFT cover}. A subshift is \emph{minimal} if it has no proper nontrivial subsystems (closed $G$-invariant subsets).

The following questions are posed in a recent paper of Michal Doucha, Julien Melleray and Todor Tsankov \cite{DoMeTs25} (they are respectively Question~7.18 (i) and (ii)).

\begin{question}
\label{q:MinimalSFTOnNonFG}
Is there a nontrivial minimal subshift of finite type on a group that is not finitely generated?
\end{question}

\begin{question}
\label{q:MinimalSoficOnNonFG}
Is there a nontrivial minimal sofic shift on a group that is not finitely generated?
\end{question}

In the following sections, we outline two somewhat disjoint motivations for Question~\ref{q:MinimalSFTOnNonFG} and Question~\ref{q:MinimalSoficOnNonFG}, the first one from the perspective of generic dynamical systems (this is how \cite{DoMeTs25} was lead to the questions), and the second an ``a posteriori'' motivation from symbolic dynamics, which the author finds equally compelling.

We solve the second question in the present paper:

\begin{theorem}
\label{thm:Main}
There is a group of the form $(F_4 \times F_2) \rtimes F_{\infty}$, which admits an infinite minimal sofic shift.
\end{theorem}

Here, $F_k$ is a free group on $k$ free generators. The $F_{\infty}$-action on $F_4 \times F_2$ is by conjugation, by suitable elements of the subgroup $F_4 \times \{1_{F_2}\} \leq F_4 \times F_2$.

We of course obtain infinitely many examples of groups where such minimal sofic shifts exist, but the assumptions are technical. The interested reader should see Lemma~\ref{lem:ABCD} for the precise requirements (and Theorem~\ref{thm:MainTheorem} for how to use them to deduce the theorem above)

It is known that not all groups will work: In \cite[Proposition~7.16]{DoMeTs25}, it is shown that infinite minimal sofic shifts cannot exist on non-finitely generated groups that are either amenable or locally finite. It is shown in \cite[Theorem~5.5]{Do24} that direct products and free products cannot be used to construct an example.

The proof of Theorem~\ref{thm:Main} combines two strong tools: high transitivity properties of Thompson's $V$ \cite{BuClRi17,BeFoHyZa25}, and strong black-box techniques for constructing sofic shifts \cite{BaSaSa21,Ba19}.

The structure of the proof is as follows. We first explain the conditions for when the free extension of a minimal system to a semidirect product stays minimal. Namely, this happens if and only if the system is topologically disjoint (in the sense of \cite{Fu67}) from the family of its automorphic conjugates realized in the semidirect product. We refer to this property as minimal $\Phi$-joinings (for $\Phi$ a class of automorphisms).

We then show that the minimal $\Phi$-joinings property can be realized by using an action with sufficiently strong transitivity properties, specifically we can use the standard action of Thompson's $V$. Finally, to make the subshift sofic, we use the simulation theorem from \cite{BaSaSa21} (or \cite{Ba19}), which morally states that all subshifts (more precisely, all up to the obvious computability restriction) are sofic on certain large groups.

Going back to Question~\ref{q:MinimalSFTOnNonFG}, our suspicion is that the answer is positive. An interesting concrete question that we cannot solve, and which may be of independent interest, is the following:

\begin{question}
Let $F_2 \curvearrowright \partial F_2$ be the boundary action (the SFT where in the standard Cayley graph of $F_2$ we orient edges with constant out-degree $1$). Is there an embedding $F_2 \leq G$ such that the induction of $F_2 \curvearrowright \partial F_2$ to $G$ is minimal?
\end{question}

\subsection{Motivation 1}

If $G$ is a countable group, the set of topological $G$-systems (i.e.\ continuous $G$-actions) on Cantor space forms a Polish space. This gives us a notion of which dynamical properties are rare (meaning the set of systems with this property is meager) and which are common (the set is residual). Some works studying spaces of dynamical systems from this perspective (in various settings and categories) are  \cite{OxUl41,Ha44,Ju81,AlPr01,GlWe01,AlPr02,AkHuKe03,FoWe04,GlThWe06,KeRo07,AkGlWe08,Ho08,BeNiDa12,Ho12a,Kw12,PaSc23,Do24}.

One interesting question is whether a generic (dense $G_\delta$) conjugacy class exists. It was shown by Kechris and Rosendal \cite{KeRo07} that such a class exists for the group $\Z$ and by Kwiatkowska \cite{Kw12} that one exists for all finitely-generated free groups. On the other hand, Hochman \cite{Ho12} shows that such a conjugacy class does not exist for $\Z^d$, $d \geq 2$, and \cite{KeRo07} shows that one does not exist for the free group on infinitely many generators.

It was shown by Doucha in \cite{Do24} that the question of existence of a generic conjugacy class of $G$-actions of Cantor space can be stated as a problem in symbolic dynamics. We say a sofic shift $Y \subset B^G$ is \emph{projectively isolated} if there is a shift-commuting continuous surjection $\phi : X \to Y$ from an SFT $X \subset A^G$, such that for some neighborhood $U$ of $X$ in the Hausdorff metric (in the space of subshifts of $A^G$, obtained from any metric inducing the Cantor topology) every subshift in $U$ maps precisely onto $Y$ under $\phi$. The following was proved in \cite{Do24} 

\begin{theorem}
Let $G$ be a countable group. A generic conjugacy class of $G$-actions on Cantor space exists if and only if projectively isolated sofic shifts are dense in the space of $G$-subshifts.
\end{theorem}

Thus, the question of a generic conjugacy class reduces to the question of density of projectively isolated sofic shifts. With this method, it was shown that if $G$ is a free product of finitely many finite and cyclic groups, then there is a generic conjugacy class of $G$-actions on Cantor space. It was also shown in \cite{Do24} that Hochman's proof of nonexistence of generic $\Z^d$-actions can be generalized to some other groups.

Doucha, Mellerey and Tsankov made the following further observation in the non-finitely generated case, \cite[Theorem~1.7]{DoMeTs25}:

\begin{theorem}
Let $G$ be a group that is not finitely generated. Then projectively isolated sofic shifts are necessarily minimal.
\end{theorem}

By this theorem, the density of projectively isolated sofic shifts is equivalent to the density of minimal sofic shifts. The motivation for Question~\ref{q:MinimalSoficOnNonFG} is that to understand whether such shifts are dense, the first question is whether they even exist.

The above theorem also implies the following interesting special property of non-finitely generated groups, \cite[Corollary~1.8]{DoMeTs25}, which further motivates the study of generic actions of such groups:

\begin{theorem}
Let $G$ be a countable group which is not finitely generated. If there exists a comeager conjugacy class in the space of Cantor actions, then the generic element is minimal.
\end{theorem}

\subsection{Motivation 2}

Subshifts of finite type are the main object of study in the field of symbolic dynamics. The beginning of symbolic dynamics on groups can be traced to Berger's 1966 paper \cite{Be66}, where he showed that there exists a $\Z^2$-SFT where every point has free orbit. The existence of such SFTs on various groups $G$ is nowadays one of the main topics of study in symbolic dynamics, where they are often called \emph{strongly aperiodic}.

Strongly aperiodic SFTs are indeed known to exist on a large class of groups \cite{Ba19,BaSa19,SaScUg20,EsMo20,BaSaSa21,CoGoRi22,BaSa24}. These range from groups with simple geometry like Baumslag-Solitar groups \cite{EsMo20} and the lamplighter group \cite{BaSa24}, to more exotic examples like all hyperbolic groups \cite{CoGoRi22}, Thompson's $V$ \cite{BaSaSa21} and the Grigorchuk group \cite{Ba19}. They have been conjectured to always exist on finitely-generated one-ended groups with decidable word problem.

Much of the study of SFTs is performed on finitely-generated groups, for the simple reason that the defining clopen set can only ``see'' finitely many group elements, and thus it is an exercise to show that every SFT $X$ on a group $G$ that is not finitely-generated actually comes from an SFT $Y$ on a finitely-generated subgroup $H \leq G$, with independent points of $X$ on different cosets of $H$ (this is the notion of induction defined in Definition~\ref{def:Induction}).

Nevertheless, it was recently observed by Barbieri that there can be interesting interplay between the group structure and the SFT, in particular subshifts of finite type can be strongly aperiodic on groups that are not finitely generated \cite{Ba23}. This suggests that it might be interesting to look at other dynamical properties of SFTs on non-finitely-generated groups, similarly leveraging the specific way the cosets live on the group.

One such important dynamical property is minimality. This property is particularly interesting because it is nontrivial but possible to establish on many groups. In the case of $\Z^2$-SFTs, examples are given in for example \cite{Mo89}, and they are known on many other groups as well. On finitely-generated free groups, minimal SFTs are trivial to find (the boundary action is an example).

This provides another strong motivation for particularly Question~\ref{q:MinimalSFTOnNonFG}, but also Question~\ref{q:MinimalSoficOnNonFG}, as sofic shifts similarly always arise from induction from a finitely-generated subgroup (since their SFT covers do).

\section{Definitions}

If $G$ is an infinite countable group, a \emph{$G$-system} is a compact metrizable zero-dimensional space $X$ together with a continuous action of $G$. Usually the zero-dimensional space $X$ is Cantor space $\{0,1\}^\omega$ (or a set homeomorphic to it). 

Two $G$-systems $G \curvearrowright X, G \curvearrowright Y$ are \emph{isomorphic} if there is a $G$-commuting homeomorphism $\phi : X \to Y$. More generally, a $G$-commuting continuous function $\phi : X \to Y$ is called a \emph{morphism}, and if surjective then it is called a \emph{factor map}, $Y$ is a \emph{factor} of $X$, and $X$ is a \emph{cover} of $Y$.

Let $A$ be a finite set (called the \emph{alphabet}). If $G$ is a countable infinite group. The \emph{full shift} on alphabet $A$ is the Cantor space $A^G$ (under the product topology) under the continuous $G$-action $gx_h = x_{g^{-1}h}$. Note that $A^G$ is the set of functions from $G$ to $A$, and we write $x_h = x(h)$ for such functions. We denote restriction of $y \in A^G$ to a subset $H$ by $y|H$. A \emph{pattern} is $p \in A^D$ for finite $D \subset G$.

A \emph{subshift} is a topologically closed $G$-invariant set $X \subseteq A^G$. A \emph{subshift of finite type} or \emph{SFT} is a subshift of the form $\{x \in A^G \;:\; \forall g \in G: gx \notin C\}$, where $C \subseteq A^G$ is clopen. Equivalently, a subshift is defined by removing from $\{0,1\}^G$ the points that contain a translate of a pattern from some set $\mathcal{F}$ of \emph{forbidden patterns}, and SFT if $\mathcal{F}$ can be chosen finite. 

 A system is \emph{SFT covered} if some SFT covers it. It is known that a subshift is up to isomorphism precisely a $G$-system where the $G$-action is \emph{expansive} meaning 
\[ \exists \epsilon > 0: \forall x, y \in X: x \neq y \implies \exists g \in G: d(gx, gy) > \epsilon. \]
We mostly blur the difference between subshifts and expansive zero-dimensional systems. An SFT covered subshift is called \emph{sofic}.

Let $G \curvearrowright X$ be an action. The \emph{faithful quotient} of the action is $G/K$ where $K$ is the pointwise stabilizer of the action. Then we have a natural action $G/K \curvearrowright X$ by $gK \cdot x = g \cdot x$. This new action is \emph{faithful}, i.e.\ no group element fixes every point of $X$.

\section{Computable actions, SFT covers and self-simulability}

When $X$ is the Cantor space $\Omega = \{0,1\}^\omega$, a homeomorphism $f : \Omega \to \Omega$ is \emph{computable} if there exists a Turing machine $T$ such that $T^x(n) = f(x)_n$ for all $x \in \Omega$. By this we mean that $T$ computes the $n$th bit of the $f$-image of $x$ on input $n$, given oracle access to $x$. An action of a finitely-generated group $G$ on $\Omega$ is \emph{computable} if the action of every $g \in G$ is computable. Homeomorphisms that are computable are closed under composition and inversion, so it suffices to check this for a generating set of $G$. Often, computable homeomorphisms are called \emph{effective}, but we avoid this term, as it is sometimes used as a synonym for faithful.

A group $G$ is \emph{self-simulable} if every computable action of $G$ has an SFT cover. The following is the proved in \cite{BaSaSa21}:

\begin{lemma}
\label{lem:Simulable}
Let $G, H$ be infinite finitely-generated nonamenable groups. Then $G \times H$ is self-simulable.
\end{lemma}

There is also the following variant from \cite{Ba19}. If $\pi : G \to H$ is a group homomorphism, then the \emph{pullback} of an $H$-system $H \curvearrowright X$ is the $G$-system $G \curvearrowright X$ given by $g \cdot x := \phi(g) \cdot x$ (where the action on the left is the new one of $G$, and the one on the right the old one of $H$).

\begin{lemma}
\label{lem:Simulable2}
Let $G, H, K$ be infinite finitely-generated groups. Then for every computable $G$-subshift, its pullback to $G \times H \times K$ is sofic.
\end{lemma}

\section{Antidiagonal minimality}

If $G \curvearrowright \Omega$, then $G$ has a \emph{diagonal action} on the Cartesian power $\Omega^n$ by $g \cdot (x_1, \ldots, x_n) = (gx_1, \ldots, gx_n)$. The \emph{antidiagonal} is $\{(x_1, \ldots, x_n) \;:\; \forall i \neq j: x_i \neq x_j\}$.

\begin{definition}
An action of a group on Cantor space $\Omega$ is \emph{antidiagonally minimal} if its diagonal action on the antidiagonal of $\Omega^n$ is minimal for all finite~$n$.
\end{definition}

This is equivalent to the following: for any $n$, for any tuple $(x_1, \ldots, x_n) \in X^n$ of distinct points, and for any $n$-tuple of nonempty open sets $(U_1, \ldots, U_n)$, there exists $g \in G$ such that $\forall i: gx_i \in U_i$.

It can be shown that the definition implies that the action is minimal on the antidiagonal of $\Omega^\kappa$ for any cardinal $\kappa$.

\begin{lemma}
\label{lem:ADMinimal}
Every antidiagonally minimal action is minimal.
\end{lemma}

\begin{proof}
This is the case $k = 1$ of the definition.
\end{proof}

\begin{lemma}
\label{lem:ADExpansive}
Any antidiagonally minimal action is expansive.
\end{lemma}

\begin{proof}
Let $G \curvearrowright X$ be an antidiagonally minimal action. It suffices to show that for some nontrivial clopen partition $U \sqcup V = X$ ($\sqcup$ denotes disjoint union), for all $x \neq y$ there exists $g \in G$ such that $gx \in U, gy \in V$, since by zero-dimensionality the distance between $U$ and $V$ is positive. But the existence of such $g$ is immediate from antidiagonal minimality, for any choice of partition $(U, V)$.
\end{proof}

\begin{lemma}
\label{lem:ADCenterless}
Let $G \curvearrowright X$ be a faithful antidiagonally minimal action on Cantor space $X$. Then $G$ has no center.
\end{lemma}

\begin{proof}
Let $a \in G$ be nontrivial, and use faithfulness to find $x \in X$ such that $ax = y \neq x$. Observe that then $y$ is not a fixed point of $a$, so there is an open set $U \ni a$ such that $aU \cap U = \emptyset$. Consider any third point $z \in X$ and use antidiagonal minimality to obtain $b \in G$ such that $by \in U$ and $bz \in U$. If $a, b$ commute, then $bax = by \in U$ and $abz \in aU$, a contradiction since $aU \cap U = \emptyset$.
\end{proof}

\section{Thompson's $V$}

The \emph{cylinders} in Cantor space $X = \{0,1\}^\omega$ are
\[ [w] = \{x \in X\;:\; x|{\{0, \ldots, |w|-1\}} = w\}. \] They can be identified with finite words $w \in \{0,1\}^*$. A \emph{complete prefix code} is $U \subset \{0,1\}^*$ such that $\{[w] \;:\; w \in U\}$ is a partition of Cantor space. Note that by compactness, a complete prefix code is finite.

\begin{definition}
Let $A, B \subset \{0,1\}^*$ be complete prefix codes of the same cardinality, and let $\phi : A \to B$ be a bijection. Then define a homeomorphism by replacing the unique prefix in $A$ by the $\phi$-image:
\[ \forall u \in A: \forall x \in \{0,1\}^\omega: f_{\phi}(ux) = \phi(u)x. \]
Thompson's $V$ consists of precisely such homeomorphisms, for triples $U, V, \phi$. The defining action of $V$ on $\{0,1\}^\omega$ is called the \emph{natural action}.
\end{definition}

We now list several (well-known) properties of $V$ and its natural action.

\begin{lemma}
\label{lem:VEffective}
The natural action of $V$ on Cantor space is computable.
\end{lemma}

\begin{proof}
This is morally obvious, since the action is defined explicitly. We describe the concrete behavior of the Turing machine corresponding to a particular element $g$. Let $A, B \subset \{0,1\}^*$ be the sets of words as in the definition of $g$, and $\phi : A \to B$ be the bijection. On input $n$ and oracle $x$, the Turing machine reads $\max_{u' \in A} |u'|$ bits from $x$ to determine its unique prefix $u \in A$, and obtains $v = \phi(u)$ from a look-up table. If the input is $n < |v|$, the Turing machine outputs $v_n$, and if $n \geq |v|$ it reads the bit $x_{n - |v| + |u|}$ from the oracle and outputs it. Then clearly on oracle $uy$ with $u \in A$, we have 
\[ T^{uy}(0)T^{uy}(1)T^{uy}(2)\ldots = \phi(u)y \]
 as desired.
\end{proof}

\begin{lemma}
The standard action of Thompson's $V$ is antidiagonally minimal.
\end{lemma}

\begin{proof}
Let $(x_1, \ldots, x_n) \in X^n$ be a tuple of distinct points, and consider a tuple of open nonempty sets $(U_1, \ldots, U_n)$. If $k$ is large enough, then we can choose cylinders $[u_i] \subset U_i$ so that $u_i$ is not a prefix of $u_j$ for any $i \neq j$ (even if $U_i = U_j$). (I.e.\ the cylinders are disjoint.)

On the other hand, if $k$ is large enough then the $k$-prefixes of the points $x_i$ are distinct, say $x_i = v_iy_i$ where $|v_i| = k$. Now define $g \in V$ as a permutation of the $k$-prefix, by setting $gv_ix = u_ix$ for all $x \in \{0,1\}^\omega$, and extend arbitrarily to the remaining prefixes in $\{0,1\}^k$.
\end{proof}

The following lemmas are immediate from the antidiagonal minimality of the action, and respectively Lemma~\ref{lem:ADMinimal} Lemma~\ref{lem:ADExpansive} and Lemma~\ref{lem:ADCenterless}.

\begin{lemma}
\label{lem:VMinimal}
The natural action of Thompson's $V$ is minimal.
\end{lemma}

\begin{lemma}
\label{lem:VExpansive}
The natural action of Thompson's $V$ is expansive.
\end{lemma}

\begin{lemma}
\label{lem:VCenterless}
Thompson's $V$ is centerless.
\end{lemma}

\begin{lemma}
\label{lem:FreeInV}
Thompson's $V$ contains a free subgroup on two generators.
\end{lemma}

\begin{proof}
It suffices to show that it contains $\Z_2 \times \Z_3 \cong \mathrm{PSL}(2, \Z)$, as this group is well-known to contain a free group. Choose $A = [0], B = [10] \cup [11]$. Define $a \in V$ by $a(ix) = (1-i)x$ for $i \in \{0,1\}, x \in \{0,1\}^\omega$. Define $B$ as the $3$-rotation with cycle decomposition $(0x \; 10x \; 11x)$. Clearly $a$ generates $\Z_2$ and $b$ generates $\Z_3$ inside Thompson's $V$. The ping-pong lemma applies, as $aB \subset A$ and $bA \cup b^2A \subset B$, and we conclude that $\langle a, b \rangle \cong \Z_2 * \Z_3$.
\end{proof}

The following is proved in \cite{DoHa20} -- it of course implies that $V$ is $2$-generated.

\begin{lemma}
\label{lem:32}
Thompson's $V$ is $\frac32$-generated. In other words, for any nontrivial element $g \in V$, there exists $h \in V$ such that $\langle g, h \rangle = V$.
\end{lemma}

Much more is known, in particular the generators can be made to have certain finite orders \cite{ScSkWu24}.

\section{Free extension and induction}

Let $X \subset A^K$ be a subshift, and let $K \leq G$. Then we obtain a subshift $Y$ on $A^G$, often called the \emph{free extension}, by using the same forbidden patterns as for $X$. Each forbidden pattern touches only a single left coset of $K$, and in turn clearly on every left coset of $K$ we have a valid point of $X$. Thus, a point $y \in Y$ contains the same information as an element of $X^{G/K}$ by defining $\phi(y)_{rK} = r^{-1}y|K$ for some left coset representatives $r \in R$, where we suppose $1_G \in R$. Note that $\{r^{-1} \;:\; r \in R\}$ form a set of right representatives.

We then have
\[ \phi(gy)_{rK} = r^{-1}gy|K = kt^{-1}y|K = kz_{tK} \]
where $t \in R, k \in K$ are such that $r^{-1}g = kt^{-1}$. This allows us to define the free extension abstractly.

\begin{definition}
\label{def:Induction}
Let $K \curvearrowright X$ be an action (which we write as $(k, x) \mapsto k \cdot x$), and let $K \leq G$. Then the \emph{induction} of the system is the $G$-system $G \curvearrowright X^{G/K}$ such that for $g \in G, z \in X^{G/K}$, we have
\[ (g * z)_{rK} := k \cdot (z_{tK}) \]
(writing the new action as $(g, z) \mapsto g * z$) where $r^{-1}g = kt^{-1}$, and $r, t \in R$.
\end{definition}

We call this abstract variant of the free extension ``induction'', following \cite{BaSa24} (this can be seen as an analog of an induced action in representation theory). The authors of \cite{DoMeTs25} call this ``co-induction'' instead.

Let $H, K$ be groups and let $\phi : H \rightarrow \Aut(K)$ be a homomorphism. We write the input of $\phi$ as subscript. Then we define the semidirect product $G = K \rtimes_\phi H$ by the formula $(k, h) (k', h') = (k \phi_h(k'), hh')$. Typically we see $K$ and $H$ as subgroups of the group $G$ in the obvious way by $K \cong K \times \{1_H\}, H \cong \{1_K\} \times H$. It is easy to check that $K$ is normal in $G$. Note that $(k, h) = kh$ for $k \in K, h \in H$, while $hk = \phi_h(k)h$.

\begin{lemma}
Let $G = K \rtimes_{\phi} H$ be a semidirect product, and suppose $K \curvearrowright X$. Then we can take $H = R$ as a set of representatives for left cosets of $K$. Then on the induction $X^{G/K}$, $(k, h) \in G$ acts on $z \in X^{G/K}$ by the formula
\[ (kh * z)_{rK} = \phi_{r^{-1}}(k) \cdot (z_{h^{-1}r K}). \]
\end{lemma}

\begin{proof}
The action of $G$ on the induction $X^{G/K}$ is then given for $(k, h) \in G, z \in X^{G/H}$ by the formula
\[ (kh * z)_{rK} = (k' \cdot z)_{tK} \]
where $r^{-1}kh = k't^{-1}$, and $r, t \in R$.
Since 
$r^{-1}kh = \phi(r^{-1})(k)r^{-1}h$, we have $r^{-1} h = t^{-1}$ i.e.\ $t = h^{-1}r$, and $k' = \phi(r^{-1})(k)$.
\end{proof}


\begin{lemma}
\label{lem:KAction}
Let $G = K \rtimes_{\phi} H$ be a semidirect product, and suppose $K \curvearrowright X$. Then the $K$-subaction on the induction is the dynamical system with points $X^H$, where $k \in K$ acts by
\[ (k * z)_h = \phi_{h^{-1}}(k) \cdot (z_h). \]
\end{lemma}

\begin{proof}
This is obtained from the previous lemma, by identifying $G/K$ with the set of representatives $H$, dropping $h$ from the formula (considering only an action of $k \in K$), and finally renaming $r$ to $h$.
\end{proof}

In other words, when an action of $K$ is induced to a semidirect product, $K$ acts in the central copy of $X$ (in coordinate $1_H$ of $X^H$) as before, and in general acts in coordinate $h \in H$ by $\phi^{-1}_h(k)$.

\section{Minimal $\Phi$-joinings and minimality of induced systems}

Let $G \curvearrowright X_i$ be Cantor dynamical systems. We say they are \emph{disjoint} if $\prod_i X_i$ (with the diagonal action) is minimal. Of course, a family of systems with two copies of the same system $G \curvearrowright X$ is never disjoint, since the diagonal is invariant. We say a system has \emph{minimal self-joinings} \cite{Ho08} (a.k.a.\ is \emph{doubly minimal} \cite{GlWe15}) if $X \times X$ has only the diagonal as a nontrivial proper subsystem. We now define a variant of this notion, where we see the other copies of $X$ through a group automorphism (in which case we need not allow the diagonal as an invariant system).

If $\phi : G \to G$ is an automorphism, we write $X^\phi$ for the $G$-system with action $g \cdot x = \phi(g) \cdot x$.

\begin{definition}
Let $\Phi$ be a set of automorphisms of $G$. We say the system $G \curvearrowright X$ has \emph{minimal $\Phi$-joinings} if the family $\{X^\phi \;:\; \phi \in \Phi\}$ is disjoint.
\end{definition}

\begin{lemma}
Let $K \curvearrowright X$ be a Cantor dynamical system. Let $\phi : H \to \Aut(K)$ be a group monomorphism and $G = K \rtimes_\phi H$. Let $G \curvearrowright X^H$ be the induction to $G$. Let $\Phi = \{ \phi(h) \;:\; h \in H \}$. If $X$ has minimal $\Phi$-joinings, then the $K$-subaction of $G \curvearrowright X^H$ is minimal.
\end{lemma}

\begin{proof}
The $K$-action on $X^H$ is
\[ (k * z)_h = \phi_{h^{-1}}(k) \cdot (z_h). \]
In other words, $K \curvearrowright X^H$ is isomorphic to the diagonal action $\prod_{\phi \in \Phi} X_\phi$, by mapping the point $z_{h^{-1}}$ to $X_{\phi_h}$. By the assumption that $X$ has minimal $\Phi$-joinings, $\prod_\phi X_\phi$ is minimal, therefore $K \curvearrowright X^H$ is minimal.
\end{proof}

\begin{lemma}
Suppose $A, B, D$ are groups. Suppose
\begin{enumerate}
\item $A$ acts on Cantor space $X$ antidiagonally minimally,
\item $B$ acts on Cantor space $X$ with faithful quotient $G$
\item $\psi : D \to \Aut(A*B)$ is defined as follows: $\psi_d(a) = a$ for $a \in A, d \in D$, and $\psi_d(b) = \phi_d b \phi_d^{-1}$ for $b \in b, d \in D$, where $\phi : D \to B$ is a homomorphism which is injective into $G/Z(G)$.
\end{enumerate}
Then the $A*B$-subaction of $X^D$ is minimal, where $X^D$ is the induction of $A*B \curvearrowright X$ to $(A*B) \rtimes_{\psi} D$.
\end{lemma}

\begin{proof}
Take any $F \Subset D$, and consider the factor $X^F$ of $X^D$ under the action of $A*B$. The action of $A$ on $X^F$ is diagonal, while the action of $b \in B$ is given by 
\[ (b * x)_f = \psi_{f^{-1}}(b) \cdot x_f = \phi_{f^{-1}} b \phi_f \cdot x_f. \]
We need to show that any $x \in X^F$ can be moved into any open set by an $A*B$-translation. If $x_f \neq x_{f'}$ for all $f \neq f'$, this is clear, because the action of $A$ is antidiagonally minimal.

Suppose then that the tuple has an equal coordinate pair $x_f = x_{f'}$ for some $f \neq f'$. We show that by translating first by an element of $A$, and then by an element of $B$, we can separate this pair, while not introducing any new equal coordinate pairs. Let us first restrict our attention only to this pair.

We start by applying a translation by $a \in A$, using which we can take $y = ax_f = ax_{f'}$ to be in any open set in $X$. Consider then the application of $b \in B$. The pair $(y, y) = (ax_f, ax_{f'})$ is taken to the pair
\[ (\phi_{f^{-1}} b \phi_f \cdot y, \phi_{f'^{-1}} b \phi_{f'} \cdot y), \]
which we would like to be off-diagonal. Equivalently, we want $y$ to be in the support of
\[ g = \phi_{f'^{-1}} b^{-1} \phi_{f'} \phi_{f^{-1}} b \phi_f. \]
Since $y$ can be in any open set, this is possible if and only if $g$ has nonempty support. But $g$ is conjugate to 
\[ \phi_f \phi_{f'^{-1}} b^{-1} \phi_{f'} \phi_{f^{-1}} b = [\phi_{f' f^{-1}}, b], \]
which for some choice of $b$ indeed has nontrivial support, since $f \neq f'$ and thus by the assumption $\phi$ maps $f'f^{-1}$ outside the center of the computable quotient of the $B$-action.

We conclude that as long as $b$ was chosen suitably, and $y = ax_f = ax_{f'}$ was taken in a certain open set, the action of $b$ separates this pair. It is clear that by choosing $a$ suitably, we can ensure that no new coincidences are introduced by applying $b$. Namely, we have a free choice of $a$ apart from the restriction that $y$ is taken into a certain open set. Since $A$ acts antidiagonally minimally, we can freely choose small open sets where $a$ moves each of the points $x_{f''}$, in such a way that the $b$-action does not introduce any new coincidences.
\end{proof}

\begin{lemma}
\label{lem:ABCD}
Suppose $A, B, C, D$ are groups and $X$ is Cantor space. Suppose
\begin{enumerate}
\item $A, B, C$ are finitely-generated;
\item  $A$ acts on $X$ antidiagonally minimally with a computable action;
\item $B$ acts expansively with computable action on $X$ with faithful quotient $G$;
\item $C$ is either a direct product of two infinite groups, or is non-amenable;
\item the map $\psi' : D \to \Aut((A*B) \times C)$ is obtained from $\psi : D \to \Aut(A*B)$ by acting trivially on $C$, and $\psi$ is as in the previous lemma (conjugation action through $\phi : D \to B$ which is injective into $G/Z(G)$).
\end{enumerate}
Then there is a sofic shift on $(A*B) \times C$ where $C$ acts trivially, and the $A*B$-subaction on $X^D$ is minimal, where $X^D$ is the induction of $(A*B) \times C \curvearrowright X$ to $((A*B) \times C) \rtimes_{\psi'} D$.
\end{lemma}

\begin{proof}
Extend the computable action of $A*B$ on $X$ trivially to $(A*B) \times C$. This system is sofic: If $C$ is nonamenable, then by Lemma~\ref{lem:Simulable} the free extension of the $A*B$-system to $(A*B) \times C$ is sofic, and the trivial extension is clearly a subSFT of this system. If in turn $C$ is a direct product of two infinite groups, then it is sofic by Lemma~\ref{lem:Simulable2}.

Since $D$ acts trivially on $C$, we have $((A*B) \times C) \rtimes_{\psi'} D \cong ((A*B) \rtimes_{\psi} D) \times C =: H$, where $C$ acts trivially on $X$. By the previous lemma, the $A*B$-subaction is minimal on the induction to $(A*B) \rtimes_\psi D$. Therefore it also acts trivially in the trivial extension to $H$.
\end{proof}

\section{A sofic minimal subshift on a group that is not finitely-generated}

\begin{theorem}
\label{thm:MainTheorem}
There is a sofic subshift $X$ on $F_4 \times F_2$, such that the action of $F_2$ is trivial, and in its free extension in some group extension $(F_4 \times F_2) \rtimes F_2$, the subaction of $F_4$ is minimal.
\end{theorem}

\begin{proof}
Set $A = F_2, B = F_2, C = F_2, E = F_2$, and define $\psi : E \to \Aut(A*B)$ by fixing $A$ pointwise and letting $\psi_e(b) = ebe^{-1}$ for $e \in E, b \in B$. Then extend $\psi$ to $\psi' : E \to \Aut((A*B) \times C)$ by acting trivially on $C$. This defines a group $((A*B) \times C) \rtimes_\psi E$. Our final group will be the group $((A*B) \times C) \rtimes_\psi D$, for a suitable subgroup $D \leq E$.

We check the assumptions of Lemma~\ref{lem:ABCD}, and in the process make the choice of $D$ and choose the actions of $A, B$ on Cantor space $Y$. Thus we need to show that
\begin{enumerate}
\item $A, B, C$ are finitely-generated,
\item $A$ acts on $Y$ antidiagonally minimally with a computable action,
\item $B$ admits an expansive computable action on $Y$ with faithful quotient $G$, 
\item $C$ is either a direct product of two infinite groups, or is non-amenable, 
\item $\psi' : D \to A*B$ comes from a conjugation action through $\phi : D \to B$, which is injective into $G/Z(G)$.
\end{enumerate}

For the first item, obviously free groups are finitely-generated. For the second, recall that Thompson's $V$ is $2$-generated (Lemma~\ref{lem:32}), and have $A$ act by these generators on Cantor space, with the natural action of $V$. Then the action of $A$ is antidiagonally minimal by Lemma~\ref{lem:VMinimal} and is computable by Lemma~\ref{lem:VEffective}. For the third, take also $B$ to act by the natural action of $V$. This is expansive by~\ref{lem:VExpansive}, and again computable by Lemma~\ref{lem:VEffective}. The faithful quotient is then $G = V$. For the fourth, we observe that $F_2$ is non-amenable.

for the fifth, define $\phi : E \to B$ by $\phi_e = e$, the natural identification of the two copies $E$ and $B$ of $F_2$. Since $V$ is not free on the two generators (indeed, it is not a free group), of course $\phi$ collapses many elements. Now we observe that $V$ contains a copy of $F_2$ by Lemma~\ref{lem:FreeInV}, say $F \leq V$. Since $F$ is free, we can pick a section $D$ for it, so that $\phi : D \to F$ is a monomorphism. Since $V$ is centerless by Lemma~\ref{lem:VCenterless}, $\phi : D \to F \leq B$ is injective into $G/Z(G)$.

Thus we have satisfied the assumptions of the lemma. We conclude from the lemma that there is a sofic shift on $F_4 \times F_2 = A*B \times C$ where $C$ acts trivially, such that the $A*B = F_4$-subaction of $X^D = X^{F_2}$ is minimal, where $X^{F_2}$ is the induction of the system $A*B \times C = F_4 \times F_2 \curvearrowright X$ to $(A*B) \times C) \rtimes_{\psi'} D = (F_4 \times F_2) \rtimes_{\psi'} F_2$.
\end{proof}

\begin{corollary}
There exists a group that is not finitely-generated, and which admits a minimal sofic shift.
\end{corollary}

\begin{proof}
Take a subgroup of $(F_4 \times F_2) \rtimes F_2$ of the form $(F_4 \times F_2) \rtimes F_{\infty}$. The sofic shift defined in the previous theorem is defined on $(F_4 \times F_2)$, and the $F_4$-subaction is minimal on the induction to $(F_4 \times F_2) \rtimes F_2$. Since $F_4 \curvearrowright (F_4 \times F_2) \rtimes F_\infty$ is a factor of $F_4 \curvearrowright (F_4 \times F_2) \rtimes F_2$, also the former is minimal. In particular, we have found a sofic shift on $(F_4 \times F_2) \rtimes F_{\infty}$ which is minimal.
\end{proof}

\bibliographystyle{plain}
\bibliography{../../../bib/bib}{}

\end{document}